\documentclass{amsart}  

\usepackage{amsrefs}
\usepackage[usenames, dvipsnames]{color}
\usepackage{amsfonts}
\usepackage{amsthm}
\usepackage{amsmath}
\usepackage{amsfonts}
\usepackage{latexsym}
\usepackage{amssymb}
\usepackage{amscd}
\usepackage[latin1]{inputenc}
\usepackage{verbatim}
\usepackage{enumerate}
\usepackage{comment}
\usepackage{mathrsfs}
\usepackage[margin=1in]{geometry}

\newtheorem{theorem}{Theorem}[section]
\newtheorem{lemma}[theorem]{Lemma}
\newtheorem{proposition}[theorem]{Proposition}
\newtheorem{corollary}[theorem]{Corollary}
\newtheorem{example}[theorem]{Example}
\newtheorem{definition}[theorem]{Definition}

\theoremstyle{definition}

\newcommand\tr{ \mbox{Tr} }

\newcommand\povm[2]{\operatorname{POVM}_{{#1}}(#2)}
\newcommand\qpm[2]{\operatorname{POVM}_{\mathcal{#1}}^1(#2)}

\newcommand{\braket}[2]{\langle #1 | #2 \rangle}

\begin{document}

\title{Bistochastic operators and quantum random variables}

\author[Sarah Plosker]{Sarah Plosker \textsuperscript{1}}

\author[Christopher Ramsey]{Christopher Ramsey \textsuperscript{1,2}}

\thanks{\textsuperscript{1}Department of Mathematics and Computer Science, Brandon University,
Brandon, MB R7A 6A9, Canada}
\thanks{\textsuperscript{2}Department of Mathematics and Statistics, MacEwan University, Edmonton, AB T5J 4S2, Canada}



\maketitle

\begin{abstract}    
Given a positive operator-valued measure $\nu$ acting on the Borel sets of a locally compact Hausdorff space $X$, with outcomes in the algebra $\mathcal B(\mathcal H)$ of all bounded operators on a (possibly infinite-dimensional) Hilbert space $\mathcal H$, one can consider $\nu$-integrable functions  $X\rightarrow \mathcal B(\mathcal H)$ that are positive quantum random variables. 
We define a seminorm on the span of such functions which in the quotient leads to a Banach space. We consider bistochastic operators acting on this space and majorization of quantum random variables is then defined with respect to these operators. 
As in classical majorization theory, we relate majorization in this context to an inequality involving all possible convex functions of a certain type. Unlike the classical setting, continuity and convergence issues arise throughout the work.    
\end{abstract}

\maketitle
\tableofcontents

\section{Introduction} 
In this work, we provide some mathematical---specifically, operator theoretic---foundational underpinnings to positive operator-valued measures and functions that are integrable with respect to these objects. It is our desire to make connections between quantum information theory and operator theory and so we first introduce the quantum context and then the pure mathematical context below.

A physical system in quantum mechanics is described mathematically by a complex separable Hilbert space $\mathcal{H}$. The observable properties of the physical system are represented by a positive, operator-valued measure (or the normalized version of such an object, called a quantum probability measure), which arises from the measurement theoretical analysis of quantum mechanics: If a quantum system undergoes a series of preparation procedures so that it is in state $\rho$, measurements are made, giving rise to a set $X$ of  outcomes, with associated $\sigma$-algebra   $\mathcal{O}(X)$   of Borel sets of $X$. One then considers a  map $\nu$ acting on $\mathcal{O}(X)$ and taking values in  the positive cone of bounded operators acting on $\mathcal H$, with the property that it is ultraweakly countably additive. The measurement outcome statistics associated to $\nu$ are given by the induced complex measure $\nu_\rho$, defined by $\nu_\rho(E) = \tr(\rho \nu(E))$ for all $E\in \mathcal O(X)$. 
See \cite{BuschOp} for a systematic presentation of the probabilistic structure of quantum mechanics. 

In more mathematical terms, we are concerned with positive operator-valued measures from the Borel sets of a locally compact Hausdorff space $X$ into $\mathcal B(\mathcal H)$ for a finite or separable Hilbert space $\mathcal H$. One considers the so-called quantum random variables with respect to $\nu$, that is, measurable functions $\psi: X\rightarrow \mathcal B(\mathcal H)$ and their integrals against $\nu$. The motivation behind this is stated in \cite{FK} as the desire for a notion of an operator-valued averaging, i.e., the quantum expected value of $\psi$. This theory has been developed in \cite{MPR} and \cite{PRLyapunov}. Other variants of this setup in the literature include \cite{pol} 
and \cite{chiri2007, chiri2010}, the latter pointing out the connection between regular operator-valued expectations and quantization maps in geometric quantization. 

One of the main goals of this paper is to bring a theorem of Komiya's \cite{Komiya} into this context. Komiya proves, for matrix majorization, that for $X, Y\in M_{m,n}(\mathbb C)$ that $X \prec Y$ if and only if $\psi(X) \leq \psi(Y)$ for every real-valued, permutation-invariant, convex function $\psi$ on $M_{m,n}(\mathbb C)$. This is shown by using the fact that the bistochastic matrices are the convex hull of the permutation matrices. Our main theorem then is that one quantum random variable is majorized by another if and only if the evaluation of the first is always smaller than or equal to the second under every real-valued, permutation-invariant, convex function. This goal is the driving force of the paper, since to define majorization one needs bistochastic operators, which in turn need a robust $L^1$ function theory. However, while this $L^1$ theory can be established in general, we are only able to introduce bistochastic operators in a much more limited way, specifically for classical measures, $\mu I_\mathcal H$, where the bistochastic operators are inherited from the classical $L^1(X,\mu)$. Even so, this variation of Komiya's theorem takes some work to develop. The more general question to define and characterize majorization in arbitrary $L^1$ spaces of quantum random variables is not discussed.

There is a rich area of study in non-commutative majorization. Besides Komiya's discussions of matrix majorization \cite{Komiya} there are many directions of majorization of operators discussed by Hiai \cite{Hiai}. Of particular interest, is the log-majorization (Araki-Lieb-Thirring inequality) of Kosaki \cite{Kosaki} on noncommutative $L^p$-spaces over arbitrary von Neumann algebras, developed by Haagerup \cite{Haagerup}. As far as the authors are aware, the $L^1$ space defined in this paper is unrelated to that of Haagerup.

The paper is organized as follows: In Section \ref{sec:prelim} we review  operator-valued measures and majorization. In Section \ref{sec:L1Norm}, we consider the span of POVM-integrable quantum random variables,  revealing certain natural candidates for a norm to be unsuitable before defining a seminorm which leads to a good choice for the $L^1$ functions into $\mathcal B(\mathcal H)$ which we show has desirable properties.  In Section \ref{sec:boundedmult}, we define a $\mathcal B(\mathcal H)$-valued bracket (a generalized inner product) between elements of our constructed $L^1$ space and the von Neumann algebra of essentially bounded quantum random variables, focusing on the issue of when multiplication of elements of these two sets yields a bounded operator. This leads to a weak topology. In Section \ref{sec:Bistoch}, prompted by classical majorization on vectors in $\mathbb{R}^n$ as well as majorization in the $L^1([0,1])$ setting, where majorization is equivalent to the existence of a stochastic or doubly stochastic matrix (depending on the context), we consider bistochastic operators on  the aforementioned Banach space.   In Section \ref{sec:MD}, we define majorization in this new context and prove the analogue of Komiya's theorem that was discussed above. 


\section{Preliminaries}\label{sec:prelim}
\subsection{Positive operator-valued measures}

Let $\mathcal B(\mathcal H)$ be the algebra of all bounded operators on $\mathcal H$ for a finite-dimensional or separable Hilbert space $\mathcal H$. 
Define $\mathcal{T}(\mathcal{H})$ as the Banach space  of all trace-class operators:  all operators  in $\mathcal B(\mathcal H)$ which have a finite trace under any orthonormal basis; the norm of this space is the trace norm $\|T\|_1=\tr\sqrt{T^*T}$. Note that the dual of $\mathcal T(\mathcal H)$ is 
$\mathcal B(\mathcal H)$:  $\mathcal T(\mathcal H)^* = \mathcal{B}(\mathcal H)$, with the duality given by $\braket{T}{A}=\tr(TA)$ for all $T\in \mathcal T(\mathcal H)$ and all $A\in \mathcal B(\mathcal H)$. The set $\mathcal S(\mathcal H)$ of all positive, trace-one trace-class operators $\rho$ (called \emph{states} or density operators) is a convex subset of $\mathcal T(\mathcal H)$. It should be emphasized here to avoid confusion that in infinite dimensions there are many states (unital, positive, linear functionals) that do not arise in this way. However, it should be noted that $\mathcal S(\mathcal H)$ is still separating for $\mathcal B(\mathcal H)$. See \cite[Chapter 1]{Holevo} for further details of these fundamental mathematical underpinnings of quantum theory. 

Throughout,   $X$ is a locally compact Hausdorff space and  $\mathcal{O}(X)$   is the $\sigma$-algebra of Borel sets of $X$.

%

\begin{definition}\cite{Larson, Paulsen, MPR}
A map $\nu : \mathcal{O}(X) \to \mathcal{B}(\mathcal{H})$ is an \emph{operator-valued measure (OVM)} if it is ultraweakly countably additive: for every countable collection $\{E_k\}_{k \in \mathbb N} \subseteq \mathcal{O}(X)$ with $E_i \cap E_j = \emptyset$ for $i \neq j$ we have
\[
\nu\left(\bigcup_{k\in \mathbb N} E_k \right) = \sum_{k \in \mathbb N} \nu(E_k)\,,
\]
where the convergence on the right side of the equation above is with respect to the ultraweak topology of $\mathcal{B}(\mathcal{H})$, that is,
\[
\tr\left(s \sum_{k=1}^n \nu(E_k)\right) \rightarrow \tr\left(s \sum_{k=1}^\infty \nu(E_k)\right), \quad \forall s\in \mathcal S(\mathcal H).
\] 
An OVM  $\nu$ is 
\begin{enumerate}[(i)]
    \item \emph{bounded} if $\sup\{\|\nu(E)\| : E\in \mathcal O(X)\} < \infty$,
    \item \emph{positive} if $\nu(E) \in \mathcal{B}(\mathcal H)_+$, for all $E\in \mathcal O(X)$; such an OVM is  called a \emph{positive operator-valued measure} (POVM),
   \item \emph{regular} if the induced complex measure $\tr(\rho\nu(\cdot))$ is regular for every $\rho\in \mathcal T(\mathcal H)$.
   \item a \emph{positive operator-valued probability measure}  or \emph{quantum probability measure} if it is positive and $\nu(X) = I_{\mathcal{H}}$. 
\end{enumerate}
\end{definition}

Note:  A POVM is necessarily bounded. Often the word \emph{observable} is used interchangeably with POVM \cite{BuschOp}, although occasionally it is used to refer to normalized POVMs \cite[Chapter 3]{Davies}. We also note that some authors consider normalization as part of the definition of a POVM \cite{BuschQT}. To avoid confusion, we use the terminology \emph{quantum probability measure} to refer to a normalized POVM, which is consistent with e.g.\ \cite{clean4, FK, FKP, MPR}. We adopt the notation $\povm{\mathcal H}{X}$ to refer to the set of all POVMs 
 $\nu : \mathcal{O}(X) \to \mathcal{B}(\mathcal{H})_+$ and $\qpm{\mathcal H}{X}$ to refer to the set of all quantum probability measures $\nu : \mathcal{O}(X) \to \mathcal{B}(\mathcal{H})_+$.

A (classical or operator-valued) measure $\omega_1$ is \emph{absolutely continuous} with respect to either a classical or operator-valued measure $\omega_2$, denoted $\omega_1\ll_{\rm ac}\omega_2$, if $\omega_1(E)=0$ whenever $\omega_2(E)=0$, where $E\in\mathcal{O}(X)$ (for classical measures, $\mathcal{O}(X)$ is typically denoted by $\Sigma$ in the literature) and 0 is interpreted as either the scalar zero or the zero operator, as applicable.   Let $\nu\in \povm{\mathcal H}{X}$.  For a fixed state $\rho\in \mathcal S(\mathcal H)$, the induced complex measure $\nu_\rho$ on $X$ is defined by $\nu_\rho(E) = \tr(\rho \nu(E))$ for all $ E\in \mathcal O(X)$. 
As discussed in \cite{MPR}, $\nu$ and $\nu_\rho$ are  mutually absolutely continuous for any full-rank $\rho\in \mathcal S(\mathcal H)$.

Let $\nu_{i,j}$ be the complex measure defined by $\nu_{i,j}(E) = \langle \nu(E)e_j,e_i\rangle, E\in \mathcal O(X)$, where $\{e_k\}$ form an orthonormal basis for $\mathcal H$. Let  $\rho\in\mathcal S(\mathcal H)$ be full-rank, that is, injective. Then  $\nu_{i,j} \ll_{\rm ac} \nu_\rho$ and so, by the classical Radon-Nikod\'ym theorem, there is a unique $\frac{d\nu_{i,j}}{d\nu_\rho} \in L_1(X, \nu_\rho)$ such that
\[
\nu_{i,j}(E) = \int_E \frac{d\nu_{i,j}}{d\nu_\rho} d\nu_\rho, \ E\in \mathcal O(X).
\]
One can then define the  {\em Radon-Nikod\'ym derivative} of $\nu$ with respect to $\nu_\rho$   to be \cite{MPR, FPS}
\[
\frac{d\nu}{d\nu_\rho} = \sum_{i,j\geq 1} \frac{d\nu_{i,j}}{d\nu_\rho} \otimes e_{i,j}.
\]

An operator-valued function  $f: X \rightarrow \mathcal{B}(\mathcal H)$ that is Borel measurable (that is,  the associated complex-valued functions $x\to \tr(s f(x))$
are Borel measurable functions for every state $s \in \mathcal S(\mathcal H)$)    is known as a {\em quantum random variable}. 
The Radon-Nikod\'ym derivative  $\frac{d\nu}{d\nu_\rho}$ is said to exist if it is a quantum random variable; i.e.\ it takes every $x$ to a bounded operator. By \cite[Corollary 2.13]{MPR}, if $\frac{d\nu}{d\nu_{\rho_0}}$ exists for some full-rank   $\rho_0\in \mathcal S(\mathcal H)$, then $\frac{d\nu}{d\nu_{\rho}}$ exists for all full-rank  $\rho\in \mathcal S(\mathcal H)$, so there is no need to specify a particular full-rank $\rho_0$.

Integrability of a quantum random variable with respect to a positive operator-valued measure is defined as follows. 
\begin{definition}\cite{FPS, MPR}
Let $\nu : \mathcal O(X) \rightarrow \mathcal B(\mathcal H)$ be a POVM such that $\frac{d\nu}{d\nu_\rho}$ exists, for a full-rank $\rho \in \mathcal S(\mathcal H)$.
A positive quantum random variable $f: X \rightarrow \mathcal{B}(\mathcal H)$ is  {\em $\nu$-integrable} if the function
\[
f_s(x) = \tr\left( s\left(\frac{d\nu}{d\nu_\rho}(x) \right)^{1/2} f(x) \left(\frac{d\nu}{d\nu_\rho}(x) \right)^{1/2} \right) 
\]
is $\nu_\rho$-integrable for every state $s\in \mathcal S(\mathcal H)$. 
If $f$ is $\nu$-integrable  then the integral of $f$ with respect to $\nu$, denoted $\int_X f d\nu$, is implicitly defined by the formula
\[
\tr\left( s\int_X f d\nu\right) = \int_X f_s d\nu_\rho.
\]
\end{definition} 

If $\nu = \mu I_\mathcal H$ for a positive complex measure $\mu$ then we know that $\frac{d\nu}{d\nu_\rho} = I_\mathcal H$ and if $f = [f_{i,j}]$ is taken with respect to an orthonormal basis in $\mathcal H$ then integration is defined entrywise:
\[
\int_X f d\nu = \left[ \int_X f_{i,j} d\mu \right].
\]

The properties of the integral are explored in depth in \cite{MPR, PRLyapunov}. 

Note that any   quantum random variable $f: X \rightarrow \mathcal{B}(\mathcal H)$ can be decomposed as the sum of four positive quantum random variables, for instance $({\rm Re} f)_+, ({\rm Re} f)_-, ({\rm Im} f)_+$, and $({\rm Im} f)_-$ but other choices as well; the definition of  $\nu$-integrable can thus be extended to arbitrary quantum random variables provided all four positive functions are $\nu$-integrable.


\subsection{Majorization}
Majorization is a preorder first defined on vectors in $\mathbb R^n$.
Let $x, y\in \mathbb{R}^n$. Then $x$ is \emph{majorized} by $y$, denoted $x\prec y$, if \begin{eqnarray*}
\sum_{j=1}^{k}x^{\downarrow}_{j}\leq \sum_{j=1}^{k}y^{\downarrow}_{j}\quad \forall k\in \{1,\dots,n-1\}
\end{eqnarray*}
with equality when $k=n$, where $x$ has been reordered so that $x^{\downarrow}_{1}\geq x^{\downarrow}_{2}\geq \cdots \geq x^{\downarrow}_{n}$ (and similarly for $y$). Alternatively, $x\prec y$ if and only if there exists a doubly stochastic matrix $S$ such that $x=Sy$ (this is a well-known result of Hardy-Littlewood-P\'{o}lya \cite[Theorem 8]{HLP}).

One can define continuous majorization in the context of functions in $L^1$:

\begin{definition}\label{def:classicdecarrange}
Let $(X, \mathcal O(X), \mu)$ be a finite positive measure space and $f\in L^1(X, \mu)$. The \emph{distribution function} of $f$ is $d_f : \mathbb R \rightarrow [0,\mu(X)]$ defined by
\[
d_f(s) = \mu(\{x : f(x) > s\})
\]
and the \emph{decreasing rearrangement} of $f$ is $f^\downarrow : [0,\mu(X)] \rightarrow \mathbb R$ defined by
\[
f^\downarrow(t) = \sup\{ s : d_f(s) \geq t\}.
\]
\end{definition}

\begin{definition}\label{def:cont} Let $(X_i, \mathcal{O}(X_i), \mu_i)$, $i=1,2$, be finite measure spaces for which $a=\mu_1(X_1)=\mu_2(X_2)$. Then $f\in L^1(X_1,\mu_1)$ is \emph{majorized} by $g\in L^1(X_2, \mu_2)$, denoted $f\prec g$, 
if 
\begin{eqnarray*}
\int_0^t f^\downarrow d x&\leq & \int_0^t g^\downarrow d x\quad \forall\, 0\leq  t\leq  a\\
\textnormal{and }\int_0^a g^\downarrow d x&=& \int_0^a f^\downarrow d x, 
\end{eqnarray*}
where integration is against Lebesgue measure.
\end{definition}
This is also called the \emph{strong spectral order} \cite{Chong}. As in the vector case, majorization is related to a certain class of operators. In particular, an operator $B: L^1(X_1,\mu_1) \rightarrow L^1(X_2,\mu_2)$ between finite measure space where $\mu_1(X_1) = \mu_2(X_2)$ is called \textit{bistochastic, doubly stochastic}, or \textit{Markov}, if 
\begin{enumerate}
    \item $B$ is positive
    \item $\displaystyle{\int_{X_2} Bf d\mu_2 = \int_{X_1} f d\mu_1}$, \ and
    \item B1 = 1
\end{enumerate}
where $1$ here refers to the constant function 1 in each of the spaces $L^1(X_i, \mu_i), i=1,2$.

The following is a combination of the well-known result by Hardy-Littlewood-P\'{o}lya \cite[Theorem 10]{HLP} extended by Chong \cite[Theorem 2.5]{Chong} and that of Ryff \cite{Ryff1965} and Day \cite{Day}

\begin{theorem}\label{thm:continuousmajorization}
Let $(X_i, \mathcal{O}(X_i), \mu_i)$, $i=1,2$, be finite measure spaces for which $\mu_1(X_1)=\mu_2(X_2)$.
If $f\in L^1(X_1, \mu_1)$ and $g\in L^1(X_2, \mu_2)$ then the following are equivalent:
\begin{itemize}
\item $f \prec g$

\item $\displaystyle{\int_{X_1} \psi(f(x))dx \leq \int_{X_2} \psi(g(x))dx}$
for all convex functions $\psi : \mathbb R\rightarrow \mathbb R$
\\
\item There is a bistochastic operator $B$ such that $Bg = f$.
\end{itemize}
\end{theorem}

\section{The L$^1$-norm}\label{sec:L1Norm}

We wish to find a generalization of the L$^1$-norm in the POVM context. Recall that $X$ is a locally compact Hausdorff space and $\mathcal H$ is finite-dimensional or separable. First, we consider the following inequalities.

\begin{lemma}
Suppose $\nu \in \povm{\mathcal H}{X}$ such that $\frac{d\nu}{d\nu_\rho}$ exists and $f : X \rightarrow \mathcal B(\mathcal H)$ is $\nu$-integrable. Then
\[
\left\| \int_X f(x) d\nu(x)\right\| \leq \int_X \|f(x)\|\left\|\frac{d\nu}{d\nu_\rho}(x)\right\| d\nu_\rho(x).
\]
Furthermore, if $\nu = \mu I$ where $\mu$ is a positive classical measure on $X$ then
\[
\left\| \int_X f(x) d\nu(x)\right\| \leq \int_X \|f(x)\| d\mu(x).
\]
\end{lemma}

\begin{proof}
Recall that the dual norm on $\mathcal B(\mathcal H)$ induced by the predual is the operator norm.
One then calculates that
\begin{align*}
    \left\| \int_X f(x) d\nu(x)\right\|
   & \ = \sup_{s\in \mathcal S(\mathcal H)}  \left|\tr\left(s\int_X f(x) d\nu(x)\right)\right|
    \\ & \ = \sup_{s\in \mathcal S(\mathcal H)} \left|\int_X \tr\left( s \left(\frac{d\nu}{d\nu_\rho}(x)\right)^{1/2}f(x)\left(\frac{d\nu}{d\nu_\rho}(x)\right)^{1/2} \right)d\nu_\rho(x)\right|
    \\ & \ \leq \sup_{s\in\mathcal  S(\mathcal H)} \int_X \left|\tr\left( s \left(\frac{d\nu}{d\nu_\rho}(x)\right)^{1/2}f(x)\left(\frac{d\nu}{d\nu_\rho}(x)\right)^{1/2} \right)\right|d\nu_\rho(x)
    \\ & \ \leq \int_X \left\| \left(\frac{d\nu}{d\nu_\rho}(x)\right)^{1/2}f(x)\left(\frac{d\nu}{d\nu_\rho}(x)\right)^{1/2}\right\|d\nu_\rho(x)
    \\ & \ \leq \int_X \|f(x)\|\left\|\left(\frac{d\nu}{d\nu_\rho}(x)\right)^{1/2}\right\|^2d\nu_\rho(x)
\end{align*}
which establishes the desired inequality. The second inequality in the statement of the lemma is immediate after observing that for any full-rank $\rho$ one has that $\nu = \mu I$ implies that $\frac{d\nu}{d\nu_\rho} = I$.
\end{proof}

In the case of self-adjoint quantum random variables we can say slightly more, but the following lemma is unlikely to be true in general.

\begin{lemma}
Suppose $\nu \in \povm{\mathcal H}{X}$ such that $\frac{d\nu}{d\nu_\rho}$ exists and $f : X \rightarrow \mathcal B(\mathcal H)$ is $\nu$-integrable and self-adjoint. Then 
\[
\left\| \int_X f(x) d\nu(x)\right\| \leq \left\| \int_X \|f(x)\| I_\mathcal H d\nu(x)\right\|.
\]
\end{lemma}
\begin{proof}
For all $x\in X$ we have that
\[
-\|f(x)\|I_\mathcal H \leq f(x) \leq \|f(x)\|I_\mathcal H.
\]
Therefore, by the comparison theorem we have that 
\[
-\int_X \|f(x)\|I_\mathcal H d\nu \leq \int_X f(x)d\nu \leq \int_X\|f(x)\|I_\mathcal H d\nu
\]
and the conclusion follows.
\end{proof}

One may believe that $\left\|\int_X \|f(x)\|I_\mathcal H d\nu(x)\right\|$ would be a good candidate for an L$^1$-norm. Namely, it reminds one of the Lebesgue-Bochner norm 
$\int_X \|f(x)\| d\mu$ on $L^1(X,\mu)\hat\otimes_\pi \mathcal B(\mathcal H)$, where $\hat\otimes_\pi$ is the projective tensor product. Indeed we can say more in finite dimensions.
\begin{lemma}\label{lemma:norm-estimate}
Suppose $\nu \in \povm{{\mathbb C}^{n}}{X}$ and $f: X \rightarrow M_n$ is $\nu$-integrable. Then
\[
\left\|\int_X \|f(x)\|I_n d\nu(x)\right\| \leq \left\|\int_X \sum_{1\leq i,j\leq n} |f_{i,j}(x)|I_n d\nu(x)\right\| \leq n^2\left\|\int_X \|f(x)\|I_n d\nu(x)\right\|.
\]
\end{lemma}
\begin{proof}
It is immediate after recalling that in $M_n$ we have that for all $x\in X$
\[
\|f(x)\| \leq \sum_{1\leq i,j\leq n} |f_{i,j}(x)| \leq n^2\|f(x)\|.
\]
\end{proof}

This shows that in finite dimensions, $\nu$-integrability is equivalent to this proposed norm being finite.
However, this quantity is too much of an overestimate in general and many good functions will not be bounded.

\begin{example}\label{example:Linftynotnormdense}
Let $X = [0,1]$, $\mathcal H$ be countably infinite dimensional, and $\nu = \mu I_\mathcal H$ where $\mu$ is Lebesgue measure. Consider $f(x) = \sum_{n\geq1} 2^{n}\chi_{(\frac{1}{2^{n}}, \frac{1}{2^{n-1}})}(x)e_{n,n}$.
This results in $\int_X f(x) d\nu = I_{\mathcal H}$ but $\left\|\int_X \|f(x)\|I_\mathcal H d\nu\right\| = \infty$.
\end{example}

A second possibility for an L$^1$-norm by analogy seems to be $\left\| \int_X |f(x)| d\nu(x) \right\|$. However, this cannot be a norm as it does not satisfy the triangle inequality, (cf.~\cite{BK} for many more oddities about the operator absolute value):

\begin{example}
Let $A=\left[\begin{matrix}1 & 0 \\ 0 &0\end{matrix}\right]$ and $B=\left[\begin{matrix}0 & 1 \\ 0 &0\end{matrix}\right]$. Then $\|A+B\| = \sqrt 2$ but $|A| + |B| = I$, 
thus 
\[
|A+B| \nleq |A| + |B| \quad  \textrm{and} \quad \|A+B\| \nleq \||A|+|B|\|.
\]
We can turn this into a counterexample to the above proposed norm by letting $X=\{0,1\}$, $\nu(0) = \nu(1) = I_2$, $f(0) = g(1) = A$  and $g(0) = f(1) = B$. Hence,
\begin{align*}
\||f(0) + g(0)| + |f(1) + g(1)|\| & = 2\|A+B\| 
\\ &\nleq  2\||A|+|B|\| 
\\ &= \||f(0)| + |f(1)|\| + \||g(0)| + |g(1)|\|. 
\end{align*}
Therefore, $\left\| \int_X |f(x)| d\nu(x)\right\|$ does not satisfy the triangle inequality.
\end{example}

We now develop an L$^1$-norm that is better adapted to POVM-integrable quantum random variables.

\begin{definition}
Let $\nu \in \povm{\mathcal H}{X}$ and define
$$\mathcal L^1_\mathcal H(X,\nu)= {\rm span}\{f:X\rightarrow \mathcal B(\mathcal H) : \nu\textrm{-integrable, positive quantum random variable}\}.$$
For every $f\in \mathcal L^1_\mathcal H(X,\nu)$ define
\[
\|f\|_1 = \inf\left\{ \left\| \int_X \sum_{k=1}^4 f_k \ d\nu \right\| : f=f_1 - f_2 + i(f_3 - f_4), f_k\in \mathcal L, f_k\geq 0, k=1,\dots, 4 \right\}.
\]
\end{definition}
We may write \ $\|f\|_{1, \nu}$ to emphasize the POVM $\nu$ that $f$ is being integrated against. 
Notice that this is a similar idea to the previous non-norm $\left\| \int_X |f(x)| d\nu(x)\right\|$ but, as we will see, with the added benefit that it actually leads to a norm.

\begin{proposition}
Let $\nu \in \povm{\mathcal H}{X}$ such that $\frac{d\nu}{d\nu_\rho}$ exists. Then $\|\cdot\|_1$ is a semi-norm on $\mathcal L^1_\mathcal H(X,\nu)$ such that $\|f^*\|_1 = \|f\|_1$.
\end{proposition}
\begin{proof}
Suppose $f,g\in \mathcal L^1_\mathcal H(X,\nu)$. For every $f_k, g_k\in \mathcal L_\nu, 1\leq k\leq 4$ such that $f_k,g_k\geq 0$, $f = f_1 - f_2 + i(f_3 - f_4)$ and $g = g_1 - g_2 + i(g_3 - g_4)$ we have that 
\begin{align*}
\|f+g\|_1 & \leq \left\| \int_X \sum_{k=1}^4 f_k+g_k \ d\nu\right\|
\\ & \leq \left\| \int_X \sum_{k=1}^4 f_k \ d\nu\right\| + \left\| \int_X \sum_{k=1}^4 g_k \ d\nu\right\|.
\end{align*}
Therefore, by taking infimums on the right, we obtain  $\|f+g\|_1 \leq \|f\|_1 + \|g\|_1$.
Lastly, we have
\begin{align*}
\|f\|_1 & =  \inf\left\{ \left\| \int_X \sum_{k=1}^4 f_k \ d\nu \right\| : f=f_1 - f_2 + i(f_3 - f_4), f_k\in \mathcal L, f_k\geq 0, 1\leq k \leq 4 \right\}
\\ & = \inf\left\{ \left\| \int_X \sum_{k=1}^4 f_k \ d\nu \right\| : f^*=f_1 - f_2 + i(f_4 - f_3), f_k\in \mathcal L, f_k\geq 0, 1\leq k\leq 4 \right\}
\\ & = \|f^*\|_1.
\end{align*}
\end{proof}

\begin{lemma}\label{lemma:equivalencetoclassical}
Let $\nu \in \povm{\mathcal H}{X}$ such that $\frac{d\nu}{d\nu_\rho}$ exists. For all $\nu$-integrable quantum random variables $f:X\rightarrow \mathcal B(\mathcal H)$
\[
\int_X f d\nu = \int_X\left( \frac{d\nu}{d\nu_\rho}^{1/2}f\frac{d\nu}{d\nu_\rho}^{1/2}\right) d\nu_\rho I_\mathcal H
\]
and so
\[
\|f\|_{1,\nu} \geq \left\| \frac{d\nu}{d\nu_\rho}^{1/2}f\frac{d\nu}{d\nu_\rho}^{1/2}\right\|_{1,\nu_\rho I_\mathcal H}.
\]
Furthermore, this is an equality if $\frac{d\nu}{d\nu_\rho}(x) \in \mathcal B(\mathcal H)^{-1}$.
\end{lemma}
\begin{proof}
For all $s\in \mathcal S(\mathcal H)$ one has that
\begin{align*}
\tr\left( s\int_X f d\nu \right) & = \int_X \tr\left(s\frac{d\nu}{d\nu_\rho}^{1/2}f\frac{d\nu}{d\nu_\rho}^{1/2}  \right)d\nu_\rho
\\ & \tr\left(s\int_X \frac{d\nu}{d\nu_\rho}^{1/2}f\frac{d\nu}{d\nu_\rho}^{1/2} d\nu_\rho I_\mathcal H\right).
\end{align*}
Thus because these states are separating, i.e.\ when  $f\in \mathcal L^1_\mathcal H(X,\nu)$ is such that $\|f\|_1 \neq 0$, then there exists a state $s\in \mathcal S(\mathcal H)$ such that $f_s \neq 0 \in L^1(X, \nu_\rho)$, the two integrals are equal.

As for the norm inequality, if $f_k\geq 0 \in \mathcal L^1_\mathcal H(X,\nu), 1\leq k\leq 4$, such that $f = f_1 - f_2 + i(f_3 - f_4)$ then by the equality above
\begin{align*}
\left\|\int_X \sum_{k=1}^4 f_k d\nu \right\|
& = \left\| \int_X \frac{d\nu}{d\nu_\rho}^{1/2}\left(\sum_{k=1}^4  f_k\right)\frac{d\nu}{d\nu_\rho}^{1/2} d\nu_\rho I_\mathcal H \right\|
\\ & = \left\| \int_X \sum_{k=1}^4 \frac{d\nu}{d\nu_\rho}^{1/2} f_k\frac{d\nu}{d\nu_\rho}^{1/2} d\nu_\rho I_\mathcal H \right\|
\\ & \geq \left\| \frac{d\nu}{d\nu_\rho}^{1/2}f_1\frac{d\nu}{d\nu_\rho}^{1/2} - \frac{d\nu}{d\nu_\rho}^{1/2}f_2\frac{d\nu}{d\nu_\rho}^{1/2} \right. \\ & \left.\quad \quad \quad \quad
+\ i\left(\frac{d\nu}{d\nu_\rho}^{1/2}f_3\frac{d\nu}{d\nu_\rho}^{1/2} - \frac{d\nu}{d\nu_\rho}^{1/2}f_4\frac{d\nu}{d\nu_\rho}^{1/2}\right) \right\|_{1,\nu_\rho I_\mathcal H}
\\ & = \left\| \frac{d\nu}{d\nu_\rho}^{1/2}f\frac{d\nu}{d\nu_\rho}^{1/2}\right\|_{1,\nu_\rho I_\mathcal H}.
\end{align*}
Taking the infimum over all possible $f_k$ we get that 
\[
\|f\|_{1,\nu} \geq  \left\| \frac{d\nu}{d\nu_\rho}^{1/2}f\frac{d\nu}{d\nu_\rho}^{1/2}\right\|_{1,\nu_\rho I_\mathcal H}.
\]
Now suppose that $g_k\in \mathcal L^1_\mathcal H(X,\nu_\rho I_\mathcal H), 1\leq k\leq 4$, such that $\frac{d\nu}{d\nu_\rho}^{1/2}f\frac{d\nu}{d\nu_\rho}^{1/2} = g_1 - g_2 + i(g_3 - g_4)$. If $\frac{d\nu}{d\nu_\rho}(x) \in \mathcal B(\mathcal H)^{-1}$ for all $x\in X$ then define $f_k = \frac{d\nu}{d\nu_\rho}^{-1/2}g_k\frac{d\nu}{d\nu_\rho}^{-1/2} \geq 0$ for $1\leq k\leq 4$. These are in $\mathcal L^1_\mathcal H(X,\nu)$ since 
\begin{align*}
\|f_k\|_{1,\nu} & = \left\| \int_X \frac{d\nu}{d\nu_\rho}^{-1/2}g_k\frac{d\nu}{d\nu_\rho}^{-1/2} d\nu \right\|
\\ & = \left\| \int_X \frac{d\nu}{d\nu_\rho}^{1/2}\left(\frac{d\nu}{d\nu_\rho}^{-1/2}g_k\frac{d\nu}{d\nu_\rho}^{-1/2}\right) \frac{d\nu}{d\nu_\rho}^{1/2} d\nu_\rho I_\mathcal H\right\|
\\ & = \left\| \int_X g_k d\nu_\rho I_\mathcal H \right\|
\\ & = \|g_k\|_{1,\nu_\rho I_\mathcal H} < \infty.
\end{align*}
Moreover,
\begin{align*}
f_1 - f_2 + i(f_3 - f_4) & = \frac{d\nu}{d\nu_\rho}^{-1/2}(g_1 - g_2 + i(g_3 - g_4))\frac{d\nu}{d\nu_\rho}^{-1/2}
\\ & = \frac{d\nu}{d\nu_\rho}^{-1/2}\left(\frac{d\nu}{d\nu_\rho}^{1/2}f\frac{d\nu}{d\nu_\rho}^{1/2}\right)\frac{d\nu}{d\nu_\rho}^{-1/2}
\\ & = f.
\end{align*}
Using the same calculations as earlier in this proof, we have 
\[
\|f\|_{1,\nu} \leq \left\| \sum_{k=1}^4 f_k \right\|_{1,\nu} = \left\|\sum_{k=1}^4 g_k\right\|_{1,\nu_\rho I_\mathcal H}
\]
and taking the infimum over all $g_k$, we obtain 
\[
\|f\|_{1,\nu} \leq \left\| \frac{d\nu}{d\nu_\rho}^{1/2}f\frac{d\nu}{d\nu_\rho}^{1/2}\right\|_{1,\nu_\rho I_\mathcal H}.
\]
\end{proof}



To further illustrate how this semi-norm behaves consider the following example which arises in \cite{BK}.
\begin{example}
Let $X=\{0,1\}$ and $\nu(i) = I_2, i=0,1$. Consider the function $f: \{0,1\} \rightarrow M_2$ given by 
\[
f(0) = \left[\begin{matrix} 4 & 4\\ 4& 4\end{matrix}\right]\quad \textrm{and} \quad f(1) = \left[\begin{matrix}3 & 0 \\ 0&-3\end{matrix}\right].
\]
Then
\[
\left\| \int_X f(x) d\nu\right\| = \|f(0) + f(1)\| =  \left\| \left[\begin{matrix} 7&4\\ 4&1\end{matrix}\right] \right\| = 9, \quad \textrm{and}
\]
\[
\left\| \int_X |f(x)| d\nu\right\| = \|f(0) + |f(1)|\| =  \left\| \left[\begin{matrix} 7&4\\ 4&7\end{matrix}\right] \right\| = 11.
\]
However, consider $f_1,f_2:\{0,1\} \rightarrow M_2$ given by
\[
f_1(0) = f(0), \ f_1(1) = \left[\begin{matrix} 4& -2\\-2& 1\end{matrix}\right], \ f_2(0) = 0_2, \ \textrm{and} \ f_2(1) = \left[\begin{matrix} 1&-2 \\ -2&4\end{matrix}\right].
\]
This gives that $f_1,f_2 \geq 0$ and $f = f_1 - f_2$ and so
\[
\|f\|_1 \leq \left\|\int_X f_1(x) + f_2(x) d\nu\right\| = \left\| f(0) + f_1(1) + f_2(1) \right\| = \left\| \left[\begin{matrix} 9&0\\ 0&9\end{matrix}\right] \right\| = 9.
\]
\end{example}

Throughout this paper an important dynamic is how this proposed 1-topology compares against various forms of the ultraweak topology. 

\begin{lemma}\label{lemma:statetopologyweaker}
Let $\nu \in \povm{\mathcal H}{X}$ such that $\frac{d\nu}{d\nu_\rho}$ exists.
If $f\in \mathcal L^1_\mathcal H(X,\nu)$ and $s\in \mathcal S(\mathcal H)$ then 
\[
\int_X |f_s|d\nu_\rho \ \leq \ \|f\|_1.
\]
\end{lemma}
\begin{proof}
Let $f_i\in \mathcal L^1_\mathcal H(X,\nu), f_i\geq 0, 1\leq i\leq 4$ such that $f= f_1-f_2 + i(f_3 - f_4)$. We have that 
\[
-f_1 - f_2 \leq f_1 - f_2 \leq f_1 + f_2 \quad \textrm{and} \quad -f_3 -f_4 \leq f_3 - f_4 \leq f_3+f_4.
\]
Recalling that $g_s = \tr\left(s\left(\frac{d\nu}{d\nu_\rho}\right)^{1/2}g\left(\frac{d\nu}{d\nu_\rho}\right)^{1/2}\right)$,
this implies that 
\[
-(f_1 + f_2)_s \leq (f_1 - f_2)_s \leq (f_1+f_2)_s \quad \textrm{and} \quad -(f_3 + f_4)_s \leq (f_3 - f_4)_s \leq (f_3 + f_4)_s
\]
and so $|(f_1-f_2)_s| \leq (f_1+f_2)_s$ and $|(f_3 -f_4)_s| \leq (f_3+f_4)_s$.
Therefore,
\begin{align*}
\int_X |f_s| d\nu_\rho & \leq \int_X |(f_1-f_2)_s| + |(f_3 - f_4)_s|\ d\nu_\rho
\\ & \leq \int_X (f_1 + f_2)_s + (f_3+f_4)_s\ d\nu_\rho
\\ & = \int_X (f_1+f_2+f_3+f_4)_s\ d\nu_\rho
\\ & = \tr\left( s\int_X f_1 + f_2 + f_3 + f_4 \ d\nu \right)
\\ & \leq \left\|\int_X f_1 + f_2 + f_3 + f_4 \ d\nu \right\|
\end{align*}
and the conclusion follows by taking the infimum over all such decompositions.
\end{proof}

In finite dimensions, with some conditions on the Radon-Nikod\'ym derivative, we get that the two semi-norms developed in this section are equivalent. In \cite{PRLyapunov} the authors introduced the von Neumann algebra of essentially bounded quantum random variables 
\begin{align*}
L^\infty_\mathcal H(X,\nu) & = \{ h:X\rightarrow \mathcal B(\mathcal H) \ \textrm{qrv} \ : \exists M\geq 0, \|h(x)\| \leq M \ \textrm{a.e wrt} \ \nu \}
\\ & = L^\infty(X,\nu_\rho) \ \bar\otimes\  \mathcal B(\mathcal H)
\end{align*}
which is needed in the following proposition and throughout the rest of the paper. Note that the norm this comes with is defined as
\[
\|f(x)\|_\infty := \Big\| \|f(x)\| \Big\|_{L^\infty(X,\nu_\rho)}
\]
since $\|f(x)\| \in L^\infty(X,\nu_\rho)$.

\begin{proposition}\label{prop:finitecomparable}
Suppose $\mathcal H= \mathbb C^n$, $\nu\in\povm{\mathcal H}{X}$ such that $\frac{d\nu}{d\nu_\rho} \in M_n$ is invertible almost everywhere ($\frac{d\nu}{d\nu_\rho} \in M_n^{-1}$ a.e.), and $\frac{d\nu}{d\nu_\rho}, \frac{d\nu}{d\nu_\rho}^{-1} \in L^\infty_\mathcal H(X,\nu)$. For $f\in \mathcal L^1_\mathcal H(X,\nu)$ self-adjoint we have
\[
\|f\|_1 \leq \left\| \int_X |f(x)| d\nu\right\| \leq \left\| \int_X \|f(x)\| I_n d\nu\right\| \leq n \left\|\frac{d\nu}{d\nu_\rho}\right\|_\infty \left\|\frac{d\nu}{d\nu_\rho}^{-1}\right\|_\infty \|f\|_1.
\]
\end{proposition}
\begin{proof}
The first two inequalities are true in general without the finite-dimensional or boundedness conditions. This is because $f$ can be written as the sum of its positive and negative parts, $f = f_+ - f_-$, and so 
\[
\|f\|_1 \leq \|f_+ + f_-\|_1 = \||f|\|_1 = \left\| \int_X |f(x)| d\nu\right\|.
\]
The second inequality follows easily since $|f(x)| \leq \|f(x)\|I_\mathcal H$.

Towards the last inequality, first consider
\begin{align*}
\left\| \int_X \|f(x)\| I_n d\nu\right\| & = \sup_{s\in \mathcal S(\mathcal H)} \tr\left( s \int_X \|f(x)\| I_n d\nu\right)
\\ & = \sup_{s\in \mathcal S(\mathcal H)} \int_X \tr\left(s \left(\frac{d\nu}{d\nu_\rho}(x)\right)^{1/2}\|f(x)\|I_n \left(\frac{d\nu}{d\nu_\rho}(x)\right)^{1/2} \right) d\nu_\rho
\\ & = \sup_{s\in \mathcal S(\mathcal H)} \int_X \|f(x)\| \tr\left(s \frac{d\nu}{d\nu_\rho}(x)\right)d\nu_\rho
\\ & \leq \left\| \frac{d\nu}{d\nu_\rho}\right\|_\infty \int_X \|f(x)\| d\nu_\rho.
\end{align*}
Here we are using the fact that $\nu$ and $\nu_\rho$ are mutually absolutely continuous, meaning $L^\infty_\mathcal H(X,\nu) = L^\infty_\mathcal H(X,\nu_\rho)$.

We will need the following nice fact about positive operators: if $A,B \geq 0$ then $-B \leq A-B \leq A$ and so $\|A-B\| \leq \max\{\|A\|,\|B\|\} \leq \|A+B\|$.
Now for each $\epsilon > 0$ there exists $f_1, f_2 \geq 0 \in L^1_\mathcal H(X,\nu)$ such that $f = f_1 - f_2$ and $\|f_1 + f_2\|_1 < \|f\|_1 + \epsilon$. For all $x\in X$ we have that 
\begin{align*}
& \|f(x)\| 
\\ & = \|f_1(x) - f_2(x)\| 
\\ &  \leq \|f_1(x) + f_2(x)\|
\\ & = \left\| \left(\frac{d\nu}{d\nu_\rho}(x)\right)^{-1/2}\left(\frac{d\nu}{d\nu_\rho}(x)\right)^{1/2}(f_1(x) + f_2(x))\left(\frac{d\nu}{d\nu_\rho}(x)\right)^{1/2}\left(\frac{d\nu}{d\nu_\rho}(x)\right)^{-1/2}  \right\|
\\ & \leq \left\|\frac{d\nu}{d\nu_\rho}^{-1} \right\|_\infty\left\|\left(\frac{d\nu}{d\nu_\rho}(x)\right)^{1/2}(f_1(x) + f_2(x))\left(\frac{d\nu}{d\nu_\rho}(x)\right)^{1/2}\right\|
\\ & \leq \left\|\frac{d\nu}{d\nu_\rho}^{-1} \right\|_\infty \tr\left(\left(\frac{d\nu}{d\nu_\rho}(x)\right)^{1/2}(f_1(x) + f_2(x))\left(\frac{d\nu}{d\nu_\rho}(x)\right)^{1/2} \right)
\\ & = \left\|\frac{d\nu}{d\nu_\rho}^{-1} \right\|_\infty\sum_{i=1}^n\tr\left(e_{i,i}\left(\frac{d\nu}{d\nu_\rho}(x)\right)^{1/2}(f_1(x) + f_2(x))\left(\frac{d\nu}{d\nu_\rho}(x)\right)^{1/2}\right)
\\ & = \left\|\frac{d\nu}{d\nu_\rho}^{-1} \right\|_\infty\sum_{i=1}^n (f_1 + f_2)_{e_{i,i}}(x).
\end{align*}
Therefore, by the comparison theorem, an earlier calculation and Lemma \ref{lemma:statetopologyweaker} we get that
\begin{align*}
\left\|\int_X \|f(x)\|I_\mathcal H d\nu\right\| & \leq \left\| \frac{d\nu}{d\nu_\rho}\right\|_\infty \int_X \|f(x)\| d\nu_\rho
\\ & \leq \left\| \frac{d\nu}{d\nu_\rho}\right\|_\infty\left\|\frac{d\nu}{d\nu_\rho}^{-1} \right\|_\infty \sum_{i=1}^n \int_X (f_1 + f_2)_{e_{i,i}} d\nu_\rho
\\ & \leq n \left\| \frac{d\nu}{d\nu_\rho}\right\|_\infty\left\|\frac{d\nu}{d\nu_\rho}^{-1} \right\|_\infty \|f_1 + f_2\|_1
\\ & < n \left\| \frac{d\nu}{d\nu_\rho}\right\|_\infty\left\|\frac{d\nu}{d\nu_\rho}^{-1} \right\|_\infty (\|f\|_1 + \epsilon).
\end{align*}
\end{proof}

Define $\mathcal I = \{f\in \mathcal L^1_\mathcal H(X,\nu) : \|f\|_1 = 0\}$ and let $L^1_\mathcal H(X,\nu) = \mathcal L^1_\mathcal H(X,\nu)/\mathcal I$. The previous lemma implies that the 1-topology on $L^1_\mathcal H(X,\nu)$ is stronger than the topology $(f_n)_s \rightarrow f_s$ for all $s\in \mathcal S(\mathcal H)$.

\begin{theorem}
$L^1_\mathcal H(X,\nu)$ is a Banach space, that is, it is complete in the 1-norm for $\nu\in \povm{\mathcal H}{X}$ where $\frac{d\nu}{d\nu_\rho}$ exists. 
\end{theorem}
\begin{proof}

Let $\{f^{(n)}\}$ be a Cauchy sequence in $L^1_\mathcal H(X,\nu)$.
There exists an increasing sequence of numbers $\{k_n\}_{n\in \mathbb N}$ such that 
\[
\|f^{(l)} - f^{(m)}\|_1 < \frac{1}{2^{n+1}}, \quad \forall l,m \geq k_n.
\]
Since $f^{(k_1)} \in L^1_\mathcal H(X,\nu)$ there exist $f_{0,i} \geq 0, 1\leq i\leq 4$ such that $f^{(k_1)} = f_{0,1} - f_{0,2} + i(f_{0,3} - f_{0,4})$ such that
\[
\left\| \int_X f_{0,1} + f_{0,2} + f_{0,3} + f_{0,4} d\nu\right\| < \|f^{(k_1)}\|_1 + 1.
\]
Similarly, $f^{(k_{n+1})} - f^{(k_n)} \in L^1_\mathcal H(X,\nu)$ and so there exists $f_{n,i} \geq 0, 1\leq i\leq 4$ such that $f^{(k_{n+1})} - f^{(k_n)} = f_{n,1} - f_{n,2} + i(f_{n,3} - f_{n,4})$ such that
\[
\left\| \int_X f_{n,1} + f_{n,2} + f_{n,3} + f_{n,4} d\nu\right\| < \|f^{(k_{n+1})} - f^{(k_n)}\|_1 + \frac{1}{2^{n+1}} < \frac{1}{2^n}.
\]
Hence, by the triangle inequality
\[
\left\| \int_X \sum_{n=0}^\infty f_{n,1} + f_{n,2} + f_{n,3} + f_{n,4} d\nu\right\| < \|f^{(k_1)}\|_1 + \sum_{n=0}^\infty \frac{1}{2^n} = \|f^{(k_1)}\|_1 + 2.
\]
Thus, $f_i := \sum_{n=0}^\infty f_{n,i}\geq 0$ is $\nu$-integrable, $1\leq i\leq 4$, and so $f := f_1 - f_2 + i(f_3 - f_4) \in \mathcal L^1_\mathcal H(X,\nu)$.

Consider now that for each $m\geq 1$, by telescoping, we have that
\begin{align*}
\|f - f^{(k_m)}\|_1 & = \left\|f^{(k_1)} + \sum_{n=1}^\infty (f^{(k_{n+1})} - f^{(k_n)}) - f^{(k_m)}\right\|_1
\\ & = \left\|f^{(k_m)} + \sum_{n=m}^\infty (f^{(k_{n+1})} - f^{(k_n)}) - f^{(k_m)}\right\|_1
\\ & = \left\| \sum_{n=m}^\infty (f^{(k_{n+1})} - f^{(k_n)}) \right\|_1
\\ & < \sum_{n=m}^\infty \frac{1}{2^n}
\\ & = \frac{1}{2^{m-1}}.
\end{align*}
Therefore, $f^{(n)} \rightarrow f$ in $\|\cdot\|_1$ and the conclusion follows.
\end{proof}


Finally in this section we relate $L^\infty_\mathcal H(X,\nu)$ and $L^1_\mathcal H(X,\nu)$.

\begin{proposition}\label{prop:denseinstatetopology}
Suppose $\frac{d\nu}{d\nu_\rho}(x) \in \mathcal B(\mathcal H)^{-1}$ for all $x\in X$ and $\frac{d\nu}{d\nu_\rho}, \frac{d\nu}{d\nu_\rho}^{-1} \in L^\infty_\mathcal H(X,\nu)$. 
There is a natural inclusion of $L^\infty_\mathcal H(X,\nu)$ in $L^1_\mathcal H(X,\nu)$ with 
\[
\|g\|_1 \leq 2\|g\|_\infty \|\nu(X)\|, \quad \forall g\in L^\infty_\mathcal H(X,\nu).
\]
Moreover, $L^\infty_\mathcal H(X,\nu)$ is dense in $L^1_\mathcal H(X,\nu)$ in the state topology, $(f_n)_s \rightarrow f_s$ for all $s\in \mathcal S(\mathcal H)$.
\end{proposition}
\begin{proof}
If $g\in L^\infty_\mathcal H(X,\nu)$ then it is easy to see that $(Re\: g)_\pm, (Im\: g)_\pm \in L^\infty_\mathcal H(X,\nu)$ as well.
Now since $\nu\in \povm{\mathcal H}{X}$ is necessarily a finite measure we have 
\begin{align*}
\|g\|_1 & \leq \|(Re\: g)_+ - (Re\: g)_-\|_1 + \|(Im\: g)_+ - (Im\: g)_-\|_1 
\\ & \leq \left\| \int_X (Re\: g)_+ + (Re\: g)_- d\nu \right\| + \left\| \int_X (Im\: g)_+ + (Im\: g)_- d\nu\right\|
\\ & \leq \left\| \int_X \|(Re\: g)_+ + (Re\: g)_-\| I_\mathcal H d\nu \right\| + \left\| \int_X \|(Im\: g)_+ + (Im\: g)_-\| I_\mathcal H d\nu\right\|
\\ & \leq \|Re\: g\|_\infty \left\| \int_X I_\mathcal H d\nu\right\| +  \|Im\: g\|_\infty \left\| \int_X I_\mathcal H d\nu\right\|
\\ & \leq 2\|g\|_\infty \|\nu(X)\|.
\end{align*}
Additionally, for $g\neq 0$ we have, by the boundedness of the Radon-Nikod\'ym derivative, that 
\[
0\neq \tilde g = \frac{d\nu}{d\nu_\rho}^{1/2}g\frac{d\nu}{d\nu_\rho}^{1/2} \in L^\infty_\mathcal H(X,\nu_\rho I_\mathcal H),
\]
which implies that $g_s \in L^\infty(X,\nu_\rho)$ for every $s\in \mathcal S(\mathcal H)$. Now if all $g_s = 0$ then $g$ would need to be 0 and so there exists an $s\in\mathcal S(\mathcal H)$ such that $g_s \neq 0$ in $L^\infty(X,\nu_\rho)$. Hence,
\[
0 < \int_X |g_s| d\nu_\rho \leq \|g\|_1
\]
by Lemma \ref{lemma:statetopologyweaker}. Therefore, $L^\infty_\mathcal H(X,\nu)$ sits inside $L^1_\mathcal H(X,\nu)$.

Now suppose $f\in L^1_\mathcal H(X,\nu)$ with $f\geq 0$. Since any class representative of $f$ is a quantum random variable from $X$ into $\mathcal B(\mathcal H)$, one can find a sequence of measurable sets $\{E_n\}_{n\in\mathbb N}$ such that
\[
\|\chi_{E_n}f\|_\infty \leq n, \quad E_n \subseteq E_{n+1}, \quad \textrm{and} \quad X = \cup_{n=1}^\infty E_n.
\]
Thus, for all $s\in \mathcal S(\mathcal H)$
\begin{align*}
(\chi_{E_n}f)_s  & = \chi_{E_n}f_s
\end{align*}
which converges to $f_s$ in the 1-norm by the monotone convergence theorem.
Therefore, since every $L^1_\mathcal H(X,\nu)$ function is the linear combination of four positive functions the conclusion is reached.
\end{proof}

This proposition implies that if $\mathcal H = \mathbb C^n$ then $L^1_\mathcal H(X,\nu) = \overline{L^\infty_\mathcal H(X,\nu)}^{\|\cdot\|_1}$. In infinite dimensions this will not be the case, for instance Example \ref{example:Linftynotnormdense} cannot be approximated by essentially bounded functions in the 1-norm.

\section{Bounded multipliers}\label{sec:boundedmult}

Although $L^\infty_\mathcal H(X,\nu)$ is not the dual space of $L^1_\mathcal H(X, \nu)$, we can think of it as a generalization of the dual space. Consider the following ``natural pairing'' or ``bracket''
\[
\langle \cdot,\cdot\rangle: L^1_\mathcal H(X, \nu)\times L^\infty_\mathcal H(X,\nu) \rightarrow \mathcal B(\mathcal H)
\]
given by
\[
\langle f,g\rangle = \int_X fg \ d\nu.
\]
The main trouble with this is that $fg$ may fail to be in $L^1_\mathcal H(X,\nu)$, or to put it another way, multiplication by $g \in L^\infty_\mathcal H(X,\nu)$ could be an unbounded operator on $L^1_\mathcal H(X,\nu)$. As the following example shows this is a problem in infinite dimensions even when $\nu = \mu I_\mathcal H$.

\begin{example}
Let $X=[0,1]$, $\mu$ Lebesgue measure on $X$, $\mathcal H = \ell^2(\mathbb N)$, and $\nu = \mu I_\mathcal H$.
Let 
\[
f(x) = \sum_{i=1}^\infty 2^{i}\chi_{(\frac{1}{2^{i}}, \frac{1}{2^{i-1}}]}(x) e_{i,i}
\]
and so
\[
\|f\|_1  = \left\|\int_0^1 f(x) d\nu \right\|  = \|I_\mathcal H\| = 1
\]
which gives that $f\in L^1_\mathcal H([0,1],\mu I_\mathcal H)$.
Now consider 
\[
g(x) = \sum_{i=1}^\infty \chi_{(\frac{1}{2^{n}}, \frac{1}{2^{n-1}}]}(x) e_{i,1} \in L^\infty_\mathcal H([0,1],\mu I_\mathcal H),
\]
since $\|g\|_\infty =1$. However, 
\[
\left\| \int_0^1 f(x)g(x) d\nu \right\| = \left\| \int_0^1 \sum_{i=1}^\infty 2^i \chi_{(\frac{1}{2^{i}}, \frac{1}{2^{i-1}}]}(x) e_{i,1}\right\| = \infty.
\]
In other words, $g$ is not a bounded right multiplier on $L^1_\mathcal H([0,1],\mu I_\mathcal H)$.
\end{example}

In general, without putting conditions on the dimension or measure, all we can say is the following:

\begin{lemma}
For all $f\in L^1_\mathcal H(X,\nu)$ and $g\in L^\infty(X,\nu_\rho)$ one has 
\[
\|f\cdot gI_\mathcal H\|_1 = \|gI_\mathcal H \cdot f\|_1 \leq 2\|f\|_1\|g\|_\infty.
\]
\end{lemma}
\begin{proof}
For every $\epsilon > 0$ there exists $f_i\in L^1_\mathcal H(X,\nu)$ such that $f_i\geq 0, 1\leq i\leq 4$ and $f = f_1 - f_2 + i(f_3 -f_4)$ such that $\|f\|_1 \leq \|f_1 +f_2+f_3+f_4\|_1 < \|f\|_1 + \epsilon$.
Now, for $g= g_1+ig_2$ where $g_1$ and $g_2$ are real-valued, we have
\begin{align*}
\|f\cdot g I_\mathcal H\|_1 &\leq \sum_{i=1}^2\|g_if_1 - g_if_2 + i(g_if_3 - g_if_4)\|_1 
\\ & \leq \sum_{i=1}^2\left\| \int_X |g_i|f_1 + |g_i|f_2 + |g_i|f_3 + |g_i|f_4 d\nu \right\|
\\ & \leq \sum_{i=1}^2\left\| \int_X \|g\|_\infty (f_1+f_2+f_3+f_4)d\nu \right\|
\\ & < 2\|g\|_\infty (\|f\|_1 + \epsilon).
\end{align*}
\end{proof}

Now we can prove a version of the Cauchy-Schwarz inequality in this context.

\begin{lemma}\label{lemma:CauchySchwarz}
If $f\in L^1_\mathcal H(X,\nu)$ and $g\in L^\infty(X,\nu_\rho)$ then
\[
\|\langle f,gI_\mathcal H\rangle\| \leq 4\|f\|_1\|g\|_\infty.
\]
\end{lemma}
\begin{proof}
Suppose $f_1,f_2\in L^1_\mathcal H(X,\nu)$ such that $f_1,f_2\geq 0$. Then 
\[
-f_1 - f_2 \leq f_1 - f_2 \leq f_1+f_2
\]
which implies that
\[
- \int_X f_1 + f_2\: d\nu \leq \int_X f_1 -f_2\: d\nu \leq \int_X f_1 + f_2\: d\nu 
\]
and so
\[
\left\|\int_X f_1 -f_2\: d\nu\right\| \leq \left\|\int_X f_1 + f_2\: d\nu\right\|.
\]
Hence, for every $f_i\in L^1_\mathcal H(X,\nu), f_i\geq 0, 1\leq i\leq 4$ such that $f=f_1-f_2 + i(f_3 -f_4)$ we have
\begin{align*}
\left\|\int_X f\: d\nu \right\| & \leq \left\|\int_X f_1 - f_2\: d\nu \right\| + \left\|\int_X f_3 - f_4\: d\nu \right\|
\\ & \leq \left\|\int_X f_1 + f_2\: d\nu \right\| + \left\|\int_X f_3 + f_4\: d\nu \right\|
\\ & \leq 2\left\|\int_X f_1 + f_2 + f_3 + f_4\: d\nu \right\|.
\end{align*}
So
\[
\left\| \int_X f\: d\nu\right\| \leq 2\|f\|_1.
\]
Therefore, by the last lemma, we have
\[
\|\langle f,gI_\mathcal H\rangle\| = \left\| \int_X fgI_\mathcal H\: d\nu\right\| \leq 2\|fgI_\mathcal H\|_1 \leq 4\|f\|_1\|g\|_\infty.
\]
\end{proof}

\begin{proposition}
Suppose $\nu = \mu I_\mathcal H \in \povm{\mathcal H}{X}$ where $\mu$ is a positive, finite measure on $X$.
If $f\in L^1_\mathcal H(X,\nu)$ and $A\in \mathcal B(\mathcal H)$ then $Af$ and $fA$ are in $L^1_\mathcal H(X,\nu)$ with
\[
\|Af\|_1 = \|fA\|_1 \leq 4\left(1+\|A\|^2\right)\|f\|_1.
\]
\end{proposition}
\begin{proof}
First suppose that $f\geq 0$. For $\lambda = \{1,-1, i, -i\}$ we have that
\[
\left[\begin{matrix}f & \bar\lambda fA^* \\ \lambda Af & AfA^*\end{matrix}\right]
= \left[\begin{matrix}I_\mathcal H &0\\ 0 &\lambda A\end{matrix}\right]\left[\begin{matrix}f & f \\ f & f\end{matrix}\right]\left[\begin{matrix}I_\mathcal H& 0\\ 0 & \lambda A\end{matrix}\right]^* \geq 0.
\]
A nice trick is that this gives
\[
0 \leq \left\langle \left[\begin{matrix}f & \bar\lambda fA^* \\ \lambda Af & AfA^*\end{matrix}\right] \left[\begin{matrix} x \\ x\end{matrix}\right], \left[\begin{matrix} x \\ x\end{matrix}\right] \right\rangle = \langle (f + \lambda Af + \bar\lambda fA^* + AfA^*)x,x \rangle
\]
for all $x\in \mathcal H$ and so $f + \lambda Af + \bar\lambda fA^* + AfA^* \geq 0$.
Now 
\begin{align*}
Re(Af) & = \frac{1}{2}(Af + fA^*)
\\ & = \frac{1}{4}(f + Af + fA^* + AfA^*) - \frac{1}{4}(f - Af - fA^* + AfA^*)
\end{align*}
and
\begin{align*}
Im(Af) & = \frac{1}{2}(-iAf + ifA^*)
\\ & = \frac{1}{4}(f -iAf + ifA^* + AfA^*) - \frac{1}{4}(f +iAf - ifA^* + AfA^*),
\end{align*}
both differences of positive operators.
Hence,
\begin{align*}
\|Af\|_1 & \leq \left\|\sum_{\lambda = \{1, -1, i, -i\}} \frac{1}{4}(f + \lambda Af + \bar\lambda fA^* + AfA^*)\right\|_1
\\ & = \|f+AfA^*\|_1
\\ & \leq \|f\|_1 + \|AfA^*\|_1
\\ & = \|f\|_1 + \sup_{s\in \mathcal S(\mathcal H)} \tr\left(s \int_X AfA^* d\nu \right)
\\ & = \|f\|_1 + \sup_{s\in \mathcal S(\mathcal H)} \int_X \tr(sAf(x)A^*)d\mu
\\ & = \|f\|_1 + \sup_{s\in \mathcal S(\mathcal H)} \int_X \tr(A^*sA f(x))d\mu
\\ & = \|f\|_1 + \sup_{s\in \mathcal S(\mathcal H)} \tr\left(A^*sA \int_X f d\nu \right)
\\ & = \|f\|_1 + \sup_{s\in \mathcal S(\mathcal H), \tr(A^*sA)\neq 0} \tr(A^*sA) \tr\left(\frac{1}{\tr(A^*sA)} A^*sA \int_X f d\nu\right)
\\ & \leq \|f\|_1 + \|A^*A\|\|f\|_1
\\ & = (1 + \|A\|^2)\|f\|_1.
\end{align*}

The general case follows easily.
In particular, letting $\epsilon > 0$, there exists $f_i \geq 0 \in L^1_\mathcal H(X,\nu), 1\leq i\leq 4$, such that $f = f_1-f_2 + i(f_3-f_4)$ and $\|\sum_{i=1}^4 f_i\|_1 < \|f\|_1 + \epsilon$.
By the above argument we have
\begin{align*}
\|Af\|_1 & \leq \|Af_1 - Af_2 + i(Af_3 - Af_4)\|_1
\\ & \leq \sum_{i=1}^4 \|Af_i\|_1
\\ & \leq \sum_{i=1}^4 (1+\|A\|^2)\|f_i\|_1
\\ & \leq 4(1+\|A\|^2)\left\|\sum_{i=1}^4 f_i\right\|_1
\\ & < 4(1+\|A\|^2)(\|f\|_1 + \epsilon).
\end{align*}
\end{proof}

\begin{corollary}\label{cor:Aderivmultiplier}
Suppose $\nu\in\povm{\mathcal H}{X}$ such that $\frac{d\nu}{d\nu_\rho}(x) \in \mathcal B(\mathcal H)^{-1}$ for all $x\in X$ and $\frac{d\nu}{d\nu_\rho}, \frac{d\nu}{d\nu_\rho}^{-1} \in L^\infty_\mathcal H(X,\nu)$. If $f\in L^1_\mathcal H(X,\nu)$ and $A\in \mathcal B(\mathcal H)$ then 
\[
\left\|\frac{d\nu}{d\nu_\rho}^{-1/2}A\frac{d\nu}{d\nu_\rho}^{1/2}f\right\|_1 = \left\|f\frac{d\nu}{d\nu_\rho}^{1/2}A\frac{d\nu}{d\nu_\rho}^{-1/2}\right\|_1 \leq 4\left(1+\|A\|^2\right)\|f\|_1.
\]
\end{corollary}
\begin{proof}
By the previous proposition and repeated uses of Lemma \ref{lemma:equivalencetoclassical} we have
\begin{align*}
\left\| \frac{d\nu}{d\nu_\rho}^{-1/2}A\frac{d\nu}{d\nu_\rho}^{1/2}f \right\|_{1, \nu}
& = \left\|\frac{d\nu}{d\nu_\rho}^{1/2} \left(\frac{d\nu}{d\nu_\rho}^{-1/2}A\frac{d\nu}{d\nu_\rho}^{1/2}f\right)\frac{d\nu}{d\nu_\rho}^{1/2} \right\|_{1, \nu_\rho I_\mathcal H}
\\ & = \left\|A\frac{d\nu}{d\nu_\rho}^{1/2}f\frac{d\nu}{d\nu_\rho}^{1/2} \right\|_{1, \nu_\rho I_\mathcal H}
\\ & \leq 4(1+\|A\|^2)\left\|\frac{d\nu}{d\nu_\rho}^{1/2}f\frac{d\nu}{d\nu_\rho}^{1/2} \right\|_{1, \nu_\rho I_\mathcal H}
\\ & = 4(1+\|A\|^2)\|f\|_{1,\nu}.
\end{align*}
\end{proof}

In finite dimensions, every multiplication operator is bounded, assuming some conditions on the Radon-Nikod\'ym derivative.

\begin{proposition}\label{prop:boundedthenallaremultipliersinfd}
Suppose $\mathcal H= \mathbb C^n$, $\nu\in\povm{\mathcal H}{X}$ such that $\frac{d\nu}{d\nu_\rho}(x) \in M_n^{-1}$ for all $x\in X$ and $\frac{d\nu}{d\nu_\rho}, \frac{d\nu}{d\nu_\rho}^{-1} \in L^\infty_\mathcal H(X,\nu)$. If $f\in L^1_\mathcal H(X,\nu)$ and $g\in L^\infty_\mathcal H(X,\nu)$ then $fg\in L^1_\mathcal H(X,\nu)$ with
\[
\|fg\|_1 \leq  n \left\|\frac{d\nu}{d\nu_\rho}\right\|_\infty \left\| \frac{d\nu}{d\nu_\rho}^{-1}\right\|_\infty \|f\|_1 \|g\|_\infty.
\]
\end{proposition}
\begin{proof}
By Proposition \ref{prop:finitecomparable} we have that
\begin{align*}
\|fg\|_1 & \leq \left\| \int_X \|f(x)g(x)\|I_n d\nu \right\|
\\ & \leq \left\| \int_X \|f(x)\|\|g\|_\infty I_n d\nu \right\|
\\ & = \|g\|_\infty \left\|\int_X \|f(x)\|I_n d\nu\right\|
\\ & \leq \|g\|_\infty n \left\|\frac{d\nu}{d\nu_\rho}\right\|_\infty \left\| \frac{d\nu}{d\nu_\rho}^{-1}\right\|_\infty \|f\|_1 < \infty. 
\end{align*}
\end{proof}

However, if the boundedness condition is dropped, multipliers can become unbounded even in finite dimensions.

\begin{example}
Let $\mathcal H = \mathbb C^2$, $X=[0,1]$, $\mu$ be Lebesgue measure and
\[
\nu = \left[ \begin{matrix} \mu & 0 \\ 0 & \sum_{i=1}^\infty 2^{i/2}\chi_{(\frac{1}{2^i},\frac{1}{2^{i-1}}]} \mu \end{matrix} \right].
\]
Now, $\nu$ is a POVM because it is positive and finite
\[
\int_0^1 \sum_{i=1}^\infty 2^{i/2}\chi_{(\frac{1}{2^i},\frac{1}{2^{i-1}}]} d\mu
= \sum_{i=1}^\infty \frac{1}{{\sqrt 2}^i} = \frac{1}{1-\frac{1}{\sqrt 2}} = \frac{\sqrt 2}{\sqrt 2 - 1} < \infty.
\]
Let $f(x) = \sum_{i=1}^\infty 2^{i/2}\chi_{(\frac{1}{2^i},\frac{1}{2^{i-1}}]}(x) e_{1,1}$ which by the above calculation gives that $f\in L^1_\mathcal H([0,1], \nu)$.
However, for $U = \left[\begin{matrix} 0& 1 \\ 1 & 0 \end{matrix}\right] \in L^\infty_\mathcal H([0,1],\nu)$ one has
\begin{align*}
\|U^*fU\|_1 & = \left\| \sum_{i=1}^\infty 2^{i/2}\chi_{(\frac{1}{2^i},\frac{1}{2^{i-1}}]}(x) e_{2,2} \right\|_1
\\ & = \left| \int_0^1 \sum_{i=1}^\infty 2^i\chi_{(\frac{1}{2^i},\frac{1}{2^{i-1}}]}(x) d\mu \right| 
\\ & = \infty.
\end{align*}
Therefore, multiplication by $U$ is not bounded on $L^1_\mathcal H([0,1], \nu)$ even though $\mathcal H = \mathbb C^2$.
\end{example}




Even though the set of bounded multipliers may be difficult to characterize we can move forward using only those arising from $L^\infty(X,\nu_\rho)$.
To this end, define the following subspace of linear functionals on $L^1_\mathcal H(X,\nu)$
\[
\mathcal F(X,\nu) = \operatorname{span}\{\tr(s\langle \cdot, g I_\mathcal H\rangle) : s\in \mathcal S(\mathcal H), g\in L^\infty(X,\nu_\rho)\}.
\]

\begin{proposition}\label{prop:fsep}
The family $\{\langle \cdot, g I_\mathcal H\rangle : g\in L^\infty(X,\nu_\rho)\}$ is separating and  $\mathcal F(X,\nu)$ is a separating subspace of linear functionals on $L^1_\mathcal H(X,\nu)$.
\end{proposition}
\begin{proof}
Suppose $f\in L^1_\mathcal H(X,\nu)$ such that $\|f\|_1 \neq 0$. 
There must exist $s\in \mathcal S(\mathcal H)$ such that $f_s \neq 0 \in L^1(X, \nu_\rho)$.
Otherwise, $\left(\frac{d\nu}{d\nu_\rho}(x)\right)^{1/2}f(x)\left(\frac{d\nu}{d\nu_\rho}(x)\right)^{1/2} = 0$ for almost all $x\in X$ with respect to $\nu$ (equally $\nu_\rho$) and so $f\equiv 0$ in $L^1_\mathcal H(X,\nu)$.

This implies that there exists a $g\in L^\infty(X,\nu_\rho)$ such that
\begin{align*}
0\neq \int_X f_sg d\nu_\rho & = \int_X \tr\left(s \left(\frac{d\nu}{d\nu_\rho}(x)\right)^{1/2}f(x)g(x)I_\mathcal H\left(\frac{d\nu}{d\nu_\rho}(x)\right)^{1/2} \right) d\nu_\rho
\\ &= \tr\left(s \int_X f(x) g(x)I_{\mathcal H} d\nu \right)
\\ &= \tr(s \langle f, gI_\mathcal H\rangle ).
\end{align*}
Therefore, the conclusion is reached.
\end{proof}

One needs to be quite careful here as it seems unlikely that this family recovers the 1-norm on $L^1_\mathcal H(X,\nu)$. However, we can show that it detects positivity.

\begin{lemma}\label{lemma:positivedetect}
Suppose $f\in L^1_\mathcal H(X,\nu)$, then $f\geq 0$ if and only if $\langle f,gI_\mathcal H\rangle \geq 0$ for all $g\in L^\infty(X,\nu_\rho)$ such that $g\geq 0$.
\end{lemma}
\begin{proof}
By the proof of the previous proposition it is easy to see that \begin{align*} f\geq 0 &\quad \Leftrightarrow \quad f_s \geq 0, & \forall s\in \mathcal S(\mathcal H) 
\\ &\quad \Leftrightarrow \quad f_sg\geq 0, & \forall s\in \mathcal S(\mathcal H), g\geq 0 \in L^\infty(X,\nu_\rho)
\\ &\quad \Leftrightarrow \quad gf \geq 0, & \forall g\geq 0 \in L^\infty(X,\nu_\rho)
\\ &\quad \Leftrightarrow \quad \langle f,gI_\mathcal H\rangle \geq 0, & \forall g\geq 0 \in L^\infty(X,\nu_\rho).
\end{align*}
\end{proof}

We say a sequence $\{f_i\}_{i\geq 1}$ in $L^1_\mathcal H(X,\nu)$ is \textit{weakly converging} to $f\in L^1_\mathcal H(X,\nu)$ if it is converging weakly with respect to the family $\mathcal F(X,\nu)$. This is the same as 
\[
\langle f_i, gI_\mathcal H\rangle \rightarrow \langle f,gI_\mathcal H\rangle, \quad \forall g\in L^\infty(X,\nu_\rho)
\]
with convergence in the ultraweak topology of $\mathcal B(\mathcal H)$.

\section{Bistochastic operators}\label{sec:Bistoch}

Throughout this section, as before, $\nu\in \povm{\mathcal H}{X}$. The following definition echoes that of the classical bistochastic operator.

\begin{definition}
A linear operator $B$ is called a {\em bistochastic operator} on $L^1_\mathcal H(X, \nu)$ if 
\begin{enumerate}
    \item $B$ is positive,
    \item $\int_X Bf d\nu = \int_X f d\nu , \quad \forall f\in L^1_\mathcal H(X,\nu)$,
    \item $BI_\mathcal H = I_\mathcal H$,
\end{enumerate}
where $I_\mathcal H$ above refers to the constant function $I_\mathcal H$ in $L^1_\mathcal H(X,\nu)$.
The set of all bistochastic operators on $L^1_\mathcal H(X, \nu)$ is denoted by $\mathfrak{B}(X,\nu)$.
\end{definition}

\begin{lemma}\label{lem:bistochasticselfadjoint}
Every bistochastic operator $B$ is self-adjoint, meaning that for every $f\in L^1_\mathcal H(X,\nu)$ we have $B(f^*) = B(f)^*$.
\end{lemma}
\begin{proof}
Let $f\in L^1_\mathcal H(X,\nu)$. There exist $f_i\in L^1_\mathcal H(X,\nu), f_i\geq 0$, $i=1, \dots, 4$, such that $f = f_1 - f_2 + i(f_3 - f_4)$. Hence,
\begin{align*}
B(f^*) & = B(f_1 - f_2 - i(f_3 - f_4))
\\ & = B(f_1) - B(f_2) - i(B(f_3) - B(f_4))
\\ & = (B(f_1) - B(f_2) + i(B(f_3) - B(f_4)))^*
\\ & = (B(f))^*.
\end{align*}
\end{proof}

\begin{lemma}\label{lem:bistoinfinitycontractive}
Every bistochastic operator takes $L^\infty_\mathcal H(X,\nu)$ to itself and is bounded in the $\infty$-norm. Furthermore, it is contractive on all self-adjoint functions.
\end{lemma}
\begin{proof}
Let $B\in \mathfrak{B}(X,\nu)$ and let $f\in L^\infty_\mathcal H(X, \nu)$ be self-adjoint. Then $\|f\|_\infty I_\mathcal H \pm f \geq 0$ almost everywhere. If $B \in \mathfrak B(X,\nu)$ then 
\[
\|f\|_\infty I_\mathcal H \pm Bf = B(\|f\|_\infty I_\mathcal H \pm f) \geq 0.
\]
Therefore, $Bf$ is essentially bounded by $\|f\|_\infty$. By the linearity of $B$, the result holds.
\end{proof}

\begin{proposition}
Every bistochastic operator is contractive with respect to the $\|\cdot\|_1$-norm.
\end{proposition}
\begin{proof}
Let $f\in L^1_\mathcal H(X,\nu)$ and $B\in \mathfrak B(X,\nu)$. For every $f_k\in L^1_\mathcal H(X,\nu), f_k\geq 0, 1\leq k\leq 4$ such that $f = f_1 -f_2 + i(f_3 - f_4)$ we have that
\begin{align*}
    \|Bf\|_1 & = \|Bf_1 - Bf_2 + i(Bf_3 - Bf_4)\|_1
    \\ & \leq \left\| \int_X \sum_{k=1}^4 Bf_k \ d\nu \right\|
    \\ & = \left\| \int_X \sum_{k=1}^4 f_k \ d\nu \right\|.
\end{align*}
Taking the infimum over all such combinations implies the conclusion.
\end{proof}

The easiest class of bistochastic operators to study are those which arise from classical bistochastic theory. 
Consider $\nu = \mu I_\mathcal H$ for some finite, positive measure $\mu$, which gives that $\frac{d\nu}{d\nu_\rho}(x) = I_\mathcal H$ and $\nu_\rho = \mu$. For this, we use $L^1(X,\mu)$ but with a compatible 1-norm $\int_X |f|_1 d\mu$ where 
\[
|f(x)|_1 = |Re\: f(x)| + |Im\: f(x)|.
\]
This implies that $L^1(X,\mu) I_\mathcal H \subset L^1(X,\nu)$.
Most classical sources seem to only consider real-valued functions so this is no different than the usual norm. Regardless, this new norm and the usual norm on $L^1$ are comparable since $|f(x)| \leq |f(x)|_1 \leq 2|f(x)|$.
Thus, with this choice of norm we have that $L^1(X,\mu)$ is norm closed.

The set of bistochastic operators on the classical $L^1(X,\mu)$ is denoted $\mathfrak B(L^1(X,\mu))$.

\begin{theorem}
If $\nu = \mu I_\mathcal H$ for some finite, positive measure $\mu$,
then every $B\in \mathfrak B(L^1(X,\mu))$ extends to a bistochastic operator in $\mathfrak B(X,\nu)$ by the formula
\[
B\left(fA\right) = B(f)A, \quad \forall f\in L^1(X,\mu), \ A\in \mathcal B(\mathcal H).
\]
\end{theorem}
\begin{proof}
Since $\mathcal H$ is separable then we can view every $f\in L^1_\mathcal H(X,\mu I_\mathcal H)$ as 
\[
f(x) = [f_{i,j}(x)] \in \mathcal B(\mathcal H)
\]
where $f_{i,j} \in L^1(X,\mu), i,j\geq 1$.
By hypothesis define $B$ on $L^1_\mathcal H(X,\mu I_\mathcal H)$ by
\[
B(f) = [B(f_{i,j})].
\]
Linearity is automatic, $B(I) = I$, and positivity follows from the fact that the action of $B$ on $f$  is akin to $B\otimes id$ and $B$ is positive. 

Now for every $f\in L^1_\mathcal H(X,\mu I_\mathcal H), f\geq 0$ we have that
\begin{align*}
\|B(f)\|_1  & = \sup_{s\in\mathcal S(\mathcal H)} \tr\left(s\int_X B(f) d\mu I_\mathcal H\right)
\\ & = \sup_{s\in\mathcal S(\mathcal H)} \int_X \tr(sB(f)) d\mu
\\ & = \sup_{s\in\mathcal S(\mathcal H)} \int_X B(\tr(sf)) d\mu
\\ & = \sup_{s\in\mathcal S(\mathcal H)} \int_X \tr(sf) d\mu
\\ & = \sup_{s\in\mathcal S(\mathcal H)} \tr\left(s\int_X f d\mu I_\mathcal H\right)
\\ & = \|f\|_1.
\end{align*}
Therefore, by the triangle inequality, $B(f) \in L^1_\mathcal H(X,\mu I_\mathcal H)$. An important step used above that will be used again in Section 6, which in fact  comes automatically with the definition of $B$, is $(Bf)_s = B(f_s)$. 
\end{proof}

We will refer to the extension developed in the previous theorem by $B$ as well and the set of such bistochastic operators as $\mathfrak B(L^1(X,\mu))$ still. We have no example of a bistochastic operator on $L^1_\mathcal H(X,\mu I_\mathcal H)$ that does not arise in this way.

\begin{corollary}\label{cor:bracketmodularity}
If $\nu = \mu I_\mathcal H$ for some finite, positive measure $\mu$,
then for every $B\in \mathfrak B(L^1(X,\mu))$, $f\in L^1(X,\mu), A\in \mathcal B(\mathcal H)$ and $g\in L^\infty(X,\mu)$
\[
\langle B(fA), g I_\mathcal H\rangle = \langle B(f),g\rangle A.
\]
\end{corollary}
\begin{proof}
This is a straightforward calculation:
\begin{align*}
\langle B(fA),g I_\mathcal H\rangle & = \langle B(f)A, gI_\mathcal H\rangle
\\ & = \int_X B(f)AgI_\mathcal H \: d\mu I_\mathcal H
\\ & = \left(\int_X B(f)g d\mu\right)A
\\ & = \langle B(f),g\rangle A.
\end{align*}
\end{proof}

Lastly, we turn to topology again. Suppose $B_i, B\in \mathfrak B(X,\nu), i\geq 1$. We say that $B_i$ is \textit{WOT-convergent} to $B$ if $B_if$ weakly converges to $Bf$ for all $f\in L^1_\mathcal H(X,\nu)$, that is
\[
\langle B_i(f), gI_\mathcal H\rangle \rightarrow \langle B(f), gI_\mathcal H\rangle, \quad \forall f\in L^1_\mathcal H(X,\nu), g\in L^\infty(X,\nu_\rho)
\]
in the ultraweak topology on $\mathcal B(\mathcal H)$.

\section{Majorization of quantum random variables}\label{sec:MD}

Recall that if $f\in L^1_\mathcal H(X,\mu I)$ and $s\in \mathcal T(\mathcal H)$ then we define $f_s\in L^1(X,\mu)$ by
\[
f_s(x) = \tr(sf(x)) \in L^1(X,\mu).
\]
We now introduce several possible majorization partial orders which relate to multivariate majorization \cite[Chapter 15]{marshallolkin} and \cite{JoeVerducci}. 

\begin{definition}\label{def:multi}
Suppose $f,g\in L^1_\mathcal H(X,\mu I)$ and are self-adjoint where $\mu$ is a finite, positive, complex measure. We say that
\begin{enumerate}
\item $f\prec g$ if there exists a bistochastic operator $B\in \mathfrak B(L^1(X,\mu))$ such that $Bg=f$,
\item $f\prec_T g$ if $f_t \prec g_t$ for all $t\in \mathcal T(\mathcal H)_{sa}$, and 
\item $f\prec_S g$ if $f_s \prec g_s$ for all $s\in \mathcal S(\mathcal H)$.
\end{enumerate}
\end{definition}


\begin{proposition}
For $f,g\in L^1_\mathcal H([0,1],\mu I)$ self-adjoint we have that
\[
f \prec g \ \ \Rightarrow \ \ f\prec_T g \ \ \Rightarrow \ \ f\prec_S g.
\]
\end{proposition}
\begin{proof}
The second implication is automatic. For the first implication assume that there is a bistochastic operator $B\in \mathfrak B(L^1_{\mathcal H})$ such that $Bg=f$.
By the definition of $B\in\mathfrak B(L^1_{\mu I_\mathcal H})$, for every $t\in \mathcal T(\mathcal H)_{sa}$ we have that $(Bf)_t = B(f_t)$. Therefore, by \cite{Ryff63} $g_t = B(f_t) \prec f_t$ for all self-adjoint $t$ and the conclusion follows.
\end{proof}

If $\mathcal H = \mathbb C$ then the converse is true by the majorization theory of $L^1$ by Theorem \ref{thm:continuousmajorization}. However, these partial orders are distinct in higher dimensions.

\begin{example}
Arising from an example of Joe and Verducci \cite{JoeVerducci}, define $f,g\in L^1_{\mathbb C^2}([0,1],\mu I)$ by
\[
f = \left[\begin{smallmatrix}1 \\ & 4\end{smallmatrix}\right]\chi_{[0,\frac{1}{2}]} + \left[\begin{smallmatrix}3 \\ & 2\end{smallmatrix}\right]\chi_{(\frac{1}{2},1]}
\quad \textrm{and} \quad g = \left[\begin{smallmatrix}1 \\ & 2\end{smallmatrix}\right]\chi_{[0,\frac{1}{2}]} + \left[\begin{smallmatrix}3 \\ & 4\end{smallmatrix}\right]\chi_{(\frac{1}{2},1]}.
\]
For every state $s\in \mathcal S(\mathbb C^2)$, whose diagonal will be non-negative numbers $a,b$, we have
\begin{align*}
f_s & = (a+4b)\chi_{[0,\frac{1}{2}]} + (3a +2b)\chi_{(\frac{1}{2},1]}, \quad \textrm{and}\\
g_s & = (a+2b)\chi_{[0,\frac{1}{2}]} + (3a +4b)\chi_{(\frac{1}{2},1]}.
\end{align*}
\cite{JoeVerducci} proves that $(a+4b, 3a+2b) \prec (a+2b, 3a+4b)$ as vectors, for all $a,b\geq0$, and so $f\prec_S g$.
However, the same paper also points out that $(1-4, 3-2) = (-3,1)$ is not majorized by $(1-2,3-4) = (-1,-1)$ and thus for $t=\left[\begin{smallmatrix}1 \\ & -1\end{smallmatrix}\right]$ we have that $f_t \nprec g_t$. Therefore, $f\prec_S g$ but $f \nprec_T g$.
\end{example}

\begin{example}
Arising from an example of Malamud \cite{Malamud}, define $f,g\in L^1_{\mathbb C^2}([0,1],\mu I)$ by
\begin{align*}
f & = \left[\begin{smallmatrix}12 \\ & 12\end{smallmatrix}\right]\chi_{[0,\frac{1}{4}]}
+ \left[\begin{smallmatrix}12 \\ & 12\end{smallmatrix}\right]\chi_{(\frac{1}{4}, \frac{1}{2}]}
+ \left[\begin{smallmatrix}5 \\ & 3\end{smallmatrix}\right]\chi_{(\frac{1}{2}, \frac{3}{4}]}
+ \left[\begin{smallmatrix}3 \\ & 5\end{smallmatrix}\right]\chi_{(\frac{3}{4}, 1]}, \quad \textrm{and}
\\ g & = \left[\begin{smallmatrix}8 \\ & 16\end{smallmatrix}\right]\chi_{[0,\frac{1}{4}]}
+ \left[\begin{smallmatrix}16 \\ & 8\end{smallmatrix}\right]\chi_{(\frac{1}{4}, \frac{1}{2}]}
+ \left[\begin{smallmatrix}0 \\ & 0\end{smallmatrix}\right]\chi_{(\frac{1}{2}, \frac{3}{4}]}
+ \left[\begin{smallmatrix}8 \\ & 8\end{smallmatrix}\right]\chi_{(\frac{3}{4}, 1]}.
\end{align*}
Using the same methodology as the previous example, \cite{Malamud} implies that $f_t \prec g_t$ for every $t\in \mathcal T(\mathbb C^2)^{sa}$. However, they also prove that there can be no bistochastic operator taking $g$ to $f$. Therefore, $f \prec_T g$ but $f\nprec g$.
\end{example}

Now to the main theory of this section. 
Recall from the introduction that Komiya \cite{Komiya} proves for $X, Y\in M_{m,n}(\mathbb C)$ that $X \prec Y$ if and only if $\psi(X) \leq \psi(Y)$ for every real-valued, permutation-invariant, convex function $\psi$ on $M_{m,n}(\mathbb C)$. The permutation matrices are significant here because the convex hull of the permutation matrices is the set of bistochastic matrices.

For the measure space $(X,\mu)$ we define $\mathcal P_{\operatorname{inv}}$ to be the set of all invertible measure-preserving transformations. In particular, this is the set of all measurable functions $\phi : X\rightarrow X$, with a measurable inverse, that satisfies the measure-preservation property: 
\[
\mu(\phi^{-1}(E)) = \mu(E), \quad \forall E\in \mathcal O(X).
\]
We use the notation $C_\phi$ to denote the right-composition operator: $C_\phi(f)=f\circ\phi$. If $\phi\in \mathcal P_{\operatorname{inv}}$ then $C_\phi$ is a bistochastic operator.

Brown \cite{Brown} has proved a similar convexity result for bistochastic operators on $L^1$ (though the paper is in the Markov operator context) under some conditions on the measure space. Namely, the convex hull $\operatorname{conv}(C_\phi:\phi\in \mathcal P_{\operatorname{inv}})$ of the composition operators  of invertible measure-preserving maps is dense in the bistochastic operators in the weak operator topology arising from $L^p$ for every $1 < p< \infty$.

\begin{proposition}\label{prop:brown}
Suppose $X$ is a product of unit intervals and $\mu$ is the corresponding product of Lebesgue measures.
If $B$ is a bistochastic operator in $\mathfrak B(L^1(X,\mu))$ then there exists a sequence of bistochastic operators $B_i\in \operatorname{conv}(C_\phi:\phi\in \mathcal P_{\operatorname{inv}})$ 
such that $B_i$ is WOT-convergent to $B$. Moreover, $\mathfrak B(L^1(X,\mu))$ is WOT-compact and convex.
\end{proposition}
\begin{proof}
$B$ is a bistochastic operator on $L^1(X,\nu)$. As mentioned above, Brown \cite[Theorem 1]{Brown} proves that 
$\operatorname{conv}(C_\phi:\phi\in \mathcal P_{\operatorname{inv}})$ is dense in the bistochastic operators in the WOT-topology, meaning that there is a sequence $B_i\in \operatorname{conv}(C_\phi:\phi\in \mathcal P_{\operatorname{inv}})$ 
such that
\[
\langle B_i h, g \rangle \rightarrow \langle Bh,g\rangle, \quad \forall h,g\in L^\infty(X,\mu).
\]

Now let $f\in L^1_\mathcal H(X,\mu I_\mathcal H), s\in \mathcal S(\mathcal H), 0\neq g\in L^\infty(X,\mu)$ and $\epsilon > 0$. By Proposition \ref{prop:denseinstatetopology}, that $L^\infty_\mathcal H(X,\mu I_\mathcal H) \simeq L^\infty(X,\mu)\ \bar\otimes\ \mathcal B(\mathcal H)$ is dense in $L^1_\mathcal H(X,\mu I_\mathcal H)$ in the state topology, there exists $h = \sum_{j=1}^m h_j A_j$ with $h_j\in L^\infty(X,\mu), 0\neq A_j\in\mathcal B(\mathcal H), 1\leq j\leq m$ such that 
\[
\|(f - h)_s\|_1 < \frac{\epsilon}{9\|g\|_\infty}.
\]

By the above we know that there exists an $N\in \mathbb N$ such that for all $i\geq N$ and $1\leq j\leq m$ we have that
\[
\left|\left\langle \big(B_i-B\big)(h_j), g\right\rangle\right| < \frac{\epsilon}{9m\|A_i\|}.
\]
Hence, for all $i\geq N$, using Lemma \ref{lemma:CauchySchwarz} and Corollary \ref{cor:bracketmodularity},  
\begin{align*}
&\left|\tr\left(s\left\langle \big(B_i - B\big)(f),g I_\mathcal H \right\rangle\right)\right| \\
& \quad \quad \leq \left|\tr\left(s\left\langle \big(B_i - B\big)(f-h), g I_\mathcal H\right\rangle\right)\right| +\left|\tr\left(s\left\langle \big(B_i - B\big)(h), gI_\mathcal H \right\rangle\right)\right|
\\ & \quad \quad\leq 4\left\|(B_i - B)((f-h)_s)\right\|_1\|g\|_\infty + \left\|\sum_{i=1}^m\langle (B_i - B)(h_j)A_j, gI_\mathcal H\rangle \right\|
\\ & \quad \quad\leq 8\|(f-h)_s\|_1\|g\|_\infty + \sum_{j=1}^m \left\|\langle (B_i - B)(h_j), gI_\mathcal H\rangle A_j\right\|
\\ & \quad \quad < \frac{8\epsilon}{9} + \sum_{j=1}^m \frac{\epsilon}{9m} = \epsilon.
\end{align*}
Therefore, $\langle B_i(f), gI_\mathcal H\rangle$ converges to $\langle B(f),gI_\mathcal H\rangle$ in the ultraweak topology of $\mathcal B(\mathcal H)$ for all $f\in L^1_\mathcal H(X,\mu I_\mathcal H)$ and $g\in L^\infty(X,\mu)$.

The last statement of the proposition easily follows by Brown's theorem \cite[Theorem 1]{Brown}, which proves that $\mathfrak B(L^1(X,\mu))$ is convex and compact in the weak topology. This is accomplished by the fact that $\mathfrak B(L^1(X,\mu))$ is a WOT-closed subset of the unit ball. Additionally, the closure of a convex hull will be convex.
\end{proof}

\begin{definition}
A real-valued convex function $\psi : L^1_\mathcal H(X, \mu I) \rightarrow \mathbb R$ is said to be {\em permutation-invariant} if for every 
$ \sigma \in \mathcal P_{\operatorname{inv}}$ we have 
\[
\psi(f\circ\sigma) = \psi(f) \quad \forall f\in L^1_\mathcal H(X, \mu I).
\]
\end{definition}

\begin{theorem}
Suppose $X$ is a product of unit intervals and $\mu$ is the corresponding product of Lebesgue measures.
Let $\tilde f, f\in L^1_\mathcal H(X, \mu I)$. Then $\tilde f \prec f$ if and only if $\psi(\tilde f) \leq \psi(f)$ for every real-valued, weakly-continuous, permutation-invariant, convex function on $L^1_\mathcal H(X, \mu I)$. 
\end{theorem}

\begin{proof}
The following proof is Komiya's argument \cite{Komiya} adapted to our context. 
Let $B\in \operatorname{conv}(C_\phi:\phi\in \mathcal P_{\operatorname{inv}})$ 
Namely, there exist $\sigma_1,\dots, \sigma_n\in P_{\operatorname{inv}}$ and $\lambda_1,\dots,\lambda_n$ positive numbers where $\sum_{i=1}^n \lambda_i = 1$ 
such that $B = \sum_{i=1}^n \lambda_i C_{\sigma_i}$.

For any real-valued, weakly-continuous, permutation-invariant, convex function $\psi$ on $L^1_\mathcal H(X, \mu I)$ we have 
\begin{align*}
    \psi(B(f)) &= \psi\left(\sum_{i=1}^n \lambda_i C_{\sigma_i}(f)\right)
    \\ & \leq \sum_{i=1}^n \lambda_i\psi(C_{\sigma_i}(f))
    \\ & = \sum_{i=1}^n \lambda_i \psi(f)
    \\ & = \psi(f).
\end{align*}

Now suppose that $B$ is an arbitrary bistochastic operator in $\mathfrak B(L^1(X,\mu))$. By Proposition \ref{prop:brown} there exists a sequence $\{B_i\}$ of bistochastic operators in the convex hull described above such that $B_i$ WOT-converges to $B$. 
Therefore, because $\psi(B_i f) \leq \psi(f)$, it follows that since $\psi$ is weakly continuous we have that $\psi(Bf) \leq \psi(f)$.

For the converse, assume that $\psi(\tilde f) \leq \psi(f)$ for all real-valued, weakly-continuous, permutation-invariant, convex function $\psi$. For every $s_1,\dots, s_m \in \mathcal S(\mathcal H)$ and $g_1,\dots, g_m \in L^\infty(X,\mu)$ we know that $\varphi(\cdot) = \sum_{i=1}^m \tr(s_i\langle \cdot,g_i\rangle)$ is an arbitrary element of $\mathcal F(X,\mu I_\mathcal H)$.
Now consider the function
\[
\psi_{\varphi}(\cdot) = \sup\{Re(\varphi\circ B(\cdot)) : B \in \mathfrak{B}(L^1(X,\mu))\}.
\]
By Lemma \ref{lemma:CauchySchwarz} 
\begin{align*}
|Re(\varphi\circ B(h))| & \leq \left|\sum_{i=1}^m \tr(s_i\langle B(h),g_i I_\mathcal H\rangle) \right|
\\ & \leq \sum_{i=1}^m |(\tr(s_i\langle B(h),g_iI_\mathcal H\rangle))| 
\\ & \leq \sum_{i=1}^m \|\langle B(h), g_i\rangle\|
\\ & \leq \sum_{i=1}^m 4\|B(h)\|_1\|g_i\|_\infty 
\\ & \leq \|h\|_1\left(4\sum_{i=1}^m \|g_i\|_\infty\right)
\end{align*}
and so $\psi_{\varphi}(h)$ exists for every $h\in L^1_\mathcal H(X,\mu I_\mathcal H)$ since it is the supremum of a bounded set of real numbers. 

Now suppose $h_i\in L^1_\mathcal H(X,\mu I_\mathcal H)$ weakly converges to $h$. Since $B\in \mathfrak B(L^1(X,\mu))$ is norm continuous then it is also weak-weak continuous in the $\mathcal F(X,\mu I_\mathcal H)$-topology. Thus, $B(h_i)$ weakly converges to $B(h)$ and $\psi_{\varphi}$ is then weakly continuous.

For any $\sigma\in \mathcal P_{\operatorname{inv}}$ we have that the map $B\mapsto BC_\sigma$ is a bijection on $\mathfrak B(L^1(X,\mu))$. Hence,
\begin{align*}
\psi_{\varphi}(C_\sigma(h)) & = \sup\{Re(\varphi\circ B(C_\sigma(h))) : B \in \mathfrak{B}(L^1(X,\mu))\}
\\ & = \psi_{\varphi}(h)
\end{align*}
and $\psi_{\varphi}$ is permutation invariant.

Lastly, $\psi_\varphi$, being the supremum of a family of linear functions, is sublinear and so is convex. Thus, by assumption we have that $\psi_{\varphi}(\tilde f) \leq \psi_{\varphi}(f)$ for every $\varphi \in \mathcal F(X,\mu I_\mathcal H)$.

By contradiction, assume that $\tilde f\neq B(f)$ for every choice of $B\in \mathfrak{B}(L^1(X,\mu))$. From the last proposition we know that $\{B(f) : B\in \mathfrak B(L^1(X,\mu))\}$ is weakly compact in the $\mathcal F(X,\mu I_\mathcal H)$-topology as well as convex. Hence, the Hahn-Banach Separation Theorem implies that there exists $\varphi \in \mathcal F(X,\mu I_\mathcal H)$ and $t\in \mathbb R$ such that
\[
Re(\varphi(f)) > t > Re(\varphi(B(f))), \ \forall B\in \mathfrak{B}(L^1(X,\mu)).
\]
Therefore,
\[
\psi_{\varphi}(\tilde f) \geq Re(\varphi(\tilde f)) > \psi_{\varphi}(f),
\]
a contradiction.
\end{proof}

\section*{Acknowledgements}
 S.P.\ was supported by NSERC Discovery Grant number 1174582, the Canada Foundation for Innovation (CFI) grant number 35711, and the Canada Research Chairs (CRC) Program grant number 231250. C.R. was supported by NSERC Discovery Grant 2019-05430. The authors would like to thank the anonymous reviewer for their helpful comments and suggestions.



\end{document}